\documentclass{amsart}

\theoremstyle{definition}
\newtheorem{theorem}{Theorem}[section]
\newtheorem{lemma}[theorem]{Lemma}
\newtheorem{corollary}[theorem]{Corollary}
\newtheorem{conjecture}[theorem]{Conjecture}
\newtheorem{note}[theorem]{Note}
\newtheorem{prop}[theorem]{Proposition}

\newtheorem{definition}[theorem]{Definition}
\newtheorem{example}[theorem]{Example}

\theoremstyle{remark}

\numberwithin{equation}{section}

\reversemarginpar

\newcommand{\realpart}{\mathop{\rm Re}\nolimits}

\newcommand{\ba}{\begin{eqnarray}}
\newcommand{\ea}{\end{eqnarray}}

\newcommand{\lf}{\left\lfloor}
\newcommand{\rf}{\right\rfloor}

\newcommand{\pFq}[5]{\ensuremath{{}_{#1}F_{#2} \left( \genfrac{}{}{0pt}{}{#3}
{#4} \bigg| {#5} \right)}}

\begin{document}

\title[Hypergeometric zeta function]
{Recursion rules for the hypergeometric zeta function}

\author[A. Byrnes]{Alyssa Byrnes}
\address{Department of Mathematics,
Tulane University, New Orleans, LA 70118}
\email{abyrnes1@tulane.edu}

\author[L Jiu]{Lin Jiu}
\address{Department of Mathematics,
Tulane University, New Orleans, LA 70118}
\email{ljiu@tulane.edu}

\author[V. Moll]{Victor H. Moll}
\address{Department of Mathematics,
Tulane University, New Orleans, LA 70118}
\email{vhm@math.tulane.edu}

\author[C. Vignat]{Christophe Vignat}
\address{Information Theory Laboratory, E.P.F.L., 1015 Lausanne, Switzerland}
\email{christophe.vignat@epfl.ch}
\address{Department of Mathematics,
Tulane University, New Orleans, LA 70118}
\email{cvignat@tulane.edu}


\date{\today}


\begin{abstract}
The hypergeometric zeta function is defined in terms of the zeros of the 
Kummer function $M(a,a+b;z)$. It is established that this function is 
an entire function of order $1$. The classical factorization theorem of 
Hadamard gives an expression as an infinite product. This provides linear 
and quadratic recurrences for the hypergeometric 
zeta function. A family of associated 
polynomials is characterized as Appell polynomials and the underlying 
distribution is given explicitly in terms of the zeros of the associated 
hypergeometric function. These properties are also given a probabilistic 
interpretation in the framework of Beta distributions.
\end{abstract}

\maketitle


\vskip 20pt 

\section{Introduction} 
\label{S:intro} 

The zeta function attached to a 
collection of non-zero complex numbers 
\newline
$\mathbb{A} = \{ a_{n} \neq 0 : \, n \in \mathbb{N} \}$, is defined by 
\begin{equation}
\zeta_{\mathbb{A}}(s) = \sum_{n=1}^{\infty} \frac{1}{a_{n}^{s}}, 
\text{ for } \realpart{s} > c. 
\end{equation}
\noindent
The most common choice of sequences $\mathbb{A}$ includes those coming from 
the zeros of a given function $f$:
\begin{equation}
\mathbb{A}(f) = \{ z \in \mathbb{C}: \, f(z) = 0 \} = 
\{ z_{n} \in \mathbb{C}: f(z_{n}) = 0, \quad  \, n \in 
\mathbb{N} \},
\end{equation}
\noindent
to produce the associated zeta function 
\begin{equation}
\zeta_{f}(s) = \sum_{n=1}^{\infty} \frac{1}{z_{n}^{s}}.
\end{equation}

The prototypical example is the classical Riemann zeta function 
\begin{equation}
\zeta(s)  = \sum_{n=1}^{\infty} \frac{1}{n^{s}},
\label{zeta-def}
\end{equation}
\noindent
coming from (half of) the zeros 
$\mathbb{A} = \{ z_{n} = n > 0 \}$ of the function 
$\begin{displaystyle} f(z) = \frac{\sin \pi z}{\pi z} \end{displaystyle}$. 

The literature contains a variety of zeta functions $\zeta_{\mathbb{A}}$ and 
their study is concentrated in reproducing the basic properties of 
\eqref{zeta-def}. For example $\zeta(s)$, originally defined in the 
half-plane $\realpart{s} > 1$, admits a meromorphic extension to the 
complex plane, with a single pole at $s=1$. Moreover, the function $\zeta(s)$ 
admits special values
\begin{equation}
\zeta(2n) = (-1)^{n+1} \frac{(2 \pi)^{2n}}{2(2n)!} B_{2n}
\end{equation}
\noindent
where the Bernoulli numbers $B_{n}$ are defined by the generating function 
\begin{equation}
\frac{t}{e^{t}-1} = \sum_{n=0}^{\infty} B_{n} \frac{t^{n}}{n!}, \quad 
|t| < 2 \pi.
\label{gen-ber}
\end{equation}

Carlitz introduced in \cite{carlitz-1961b} coefficients $\beta_{n}$ by 
\begin{equation}
\frac{x^{2}}{e^{x}-1-x} = \sum_{n=0}^{\infty} \beta_{n}\frac{x^{n}}{n!}
\label{carlitz-61}
\end{equation}
\noindent
and stated that \textit{nothing is known about them}. Howard 
\cite{howard-1967b} used the notation \newline $A_{s} = \tfrac{1}{2}\beta_{s}$, 
and in \cite{howard-1967a} he introduced the generalization 
$A_{k,r}$ by 
\begin{equation}
\frac{x^{k}}{k!} \left( e^{x} - \sum_{s=0}^{k-1} \frac{x^{s}}{s!} \right)^{-1} 
= \sum_{r=0}^{\infty} A_{k,r} \frac{x^{r}}{r!}. 
\label{howard-gen}
\end{equation}
\noindent
These numbers satisfy the recurrence 
\begin{equation}
\sum_{r=0}^{n} \binom{n+k}{r} A_{k,r} = 0, \text{ for } n > 0
\end{equation}
\noindent 
with  $A_{k,0}=1$. It follows that $A_{k,n}$ is a rational number and some of 
their arithmetical properties are reviewed in Section \ref{sec-arithmetical}.

The work presented here considers a zeta function constructed in terms of 
the \textit{Kummer function}
\begin{equation}
M(a,b;z) = \pFq11{a}{b}{z} = 
\sum_{k=0}^{\infty} \frac{(a)_{k}}{(b)_{k}} \frac{z^{k}}{k!}.
\label{kummer-1-def}
\end{equation}

The next definition introduces the main function considered here. The 
notation 
\begin{equation}
\Phi_{a,b}(z) = \pFq11{a}{a+b}{z} = M(a,a+b;z), \text{ for } a, \, b  
\in \mathbb{R}
\label{def-kummer}
\end{equation}
\noindent
is employed throughout. 

\begin{definition}
Let $a, \, b$ be positive real  numbers. The 
\textit{hypergeometric zeta function} is defined by 
\begin{equation}
\zeta^{H}_{a,b}(s) = \sum_{k=1}^{\infty} \frac{1}{z_{k;a,b}^{s}}
\text{ for } \realpart{s} > 1,
\label{hypergeom-def}
\end{equation}
\noindent
where $z_{k;a,b}$ is the sequence of complex zeros of the function 
$\Phi_{a,b}(z)$.
\end{definition}

\medskip

The special case $\Phi_{1,1}(z) = (e^{z}-1)/z$ is the reciprocal of the 
generating function for the Bernoulli numbers \eqref{gen-ber}. The 
coefficients $B_{n}^{(b)}$ are defined by
\begin{equation}
\frac{1}{\Phi_{1,b}(z)} = \sum_{n=0}^{\infty} 
B_{n}^{(b)} \frac{z^{n}}{n!}.
\label{def-Phi}
\end{equation}
\noindent
In the case $b \in \mathbb{N}$, these numbers are the coefficients 
$A_{k,r}$ defined by Howard in \eqref{howard-gen} (with $k=b$ and $r=n$). 
These numbers are discussed in  
Section \ref{sec-hyperB}. The function $\Phi_{1,2}(z)$ recently appeared  in 
\cite{deangelis-2009a} in the asymptotic expansion of $n!$. Indeed, it 
can be shown that the 
coefficients $a_{k}$ in the expansion 
\begin{equation}
n! \sim \frac{n^{n} \sqrt{2 \pi n}}{e^{n}} 
\sum_{k=0}^{\infty} \frac{a_{k}}{n^{k}} \text{ as } n \to \infty, 
\end{equation}
\noindent
are given by 
\begin{equation}
a_{k} = \frac{1}{2^{k}k!} 
\left( \frac{d}{dz} \right)^{2k} \Phi_{1,2}^{-(k+1/2)} \Big{|}_{z=0}.
\end{equation}
\noindent
K. Dilcher \cite{dilcher-2002a,dilcher-2002b} considered the zeta function 
$\zeta_{a,b}^{H}$. In particular, he established an expression for 
$\zeta_{a,b}^{H}(m)$, for $a, \, b, \, m \in \mathbb{N}$, in 
terms of the hypergeometric Bernoulli numbers
$B_{a,b}^{n}$ introduced in Section \ref{sec-generalized}.

\begin{note}
The many examples of zeta functions discussed in the literature include 
the \textit{Bessel zeta function}
\begin{equation}
\zeta_{\textit{Bes},a}(s)  = \sum_{n=1}^{\infty} \frac{1}{j_{a,n}^{s}}
\end{equation}
\noindent
where $\{ j_{a,n}\}$ are the zeros of $J_{a}(z)/z^{a}$, with $J_{a}(z)$ the 
Bessel function of the first kind
\begin{equation}
J_{a}(z) = \sum_{m=0}^{\infty} \frac{(-1)^{m}}{m! \Gamma(m + a + 1) }
\left( \frac{z}{2} \right)^{2m+a}. 
\end{equation}
\noindent
Papers considering $\zeta_{\textit{Bes},a}$ include
\cite{actor-1996a,elizalde-1995a,hawkins-1983a,stolarski-1985a}. A second 
example is the \textit{Airy-zeta function}, defined by 
\begin{equation}
\zeta_{Ai}(s) = \sum_{n=1}^{\infty} \frac{1}{a_{n}^{s}},
\end{equation}
\noindent
where $\{ a_{n} \}$ are the zeros of the Airy function 
\begin{equation}
\textit{Ai}(x) = \frac{1}{\pi} \int_{0}^{\infty} \cos \left( 
\tfrac{1}{3}t^{3} + xt \right) \, dt.
\end{equation}
\noindent 
This is considered by R. Crandall \cite{crandall-1996a} in the 
so-called \textit{quantum bouncer}. Special values include the remarkable
\begin{equation}
\zeta_{Ai}(2) = \frac{3^{5/3}}{4 \pi^{2}} \Gamma^{4} \left( \frac{2}{3} 
\right).
\end{equation}
\noindent
A third example is the zeta function
studied by A. Hassen and H. Nguyen \cite{hassen-2008a,hassen-2010a}. 
This is defined by the integral 
\begin{equation}
\zeta_{HN,b}(s) = \frac{1}{\Gamma(s+b-1)} \int_{0}^{\infty} 
\frac{x^{s+b-2} \, dx}{e^{x} - 1-z-z^{2}/2!- \cdots - z^{b-1}/(b-1)!}.
\label{zeta-has1}
\end{equation}
\end{note}

\medskip

The main results presented here include a relation among the Kummer function
$\Phi_{a,b}(z)$ and the hypergeometric zeta function $\zeta_{a,b}^{H}$ 
defined in \eqref{hypergeom-def}. This is 
\begin{equation}
\frac{\Phi_{a,b+1}(z)}{\Phi_{a,b}(z)} = 1 + \frac{a+b}{b} 
\sum_{k=1}^{\infty} \zeta_{a,b}^{H}(k+1)z^{k}
\end{equation}
\noindent
and it appears in Proposition \ref{kummer-90}. It is shown that the 
function $\zeta_{a,b}^{H}$
satisfies a couple of linear recurrence relations: 
\begin{equation}
\sum_{\ell=1}^{p} B(a+p-\ell,b) \frac{p!}{(p- \ell)!} 
\zeta_{a,b}^{H}(\ell+1) = 
- \frac{bp}{(a+b)(a+b+p)}B(a+p,b).
\label{linear-rec-1}
\end{equation}
\noindent
appearing in Theorem \ref{thm-linear} and 
\begin{equation}
(n-1)! \sum_{j=2}^{n} \frac{B_{n-j}^{(a,b)}}{(n-j)!} \zeta_{a,b}^{H}(j) = 
\frac{a}{a+b} B_{n-1}^{(a,b)} + B_{n}^{(a,b)},
\end{equation}
\noindent
established in Theorem \ref{thm-linear-2}, where $B_{n}^{(a,b)}$ are the 
so-called \textit{hypergeometric Bernoulli numbers} introduced in 
Section \ref{sec-generalized}. The third recurrence is quadratic
\begin{equation}
\sum_{k=1}^{p} \zeta_{a,b}^{H}(k+1)\zeta_{a,b}^{H}(p-k+1) = 
(a+b+p+1) \zeta_{a,b}^{H}(p+2) + \left( \frac{a-b}{a+b} \right) 
\zeta_{a,b}^{H}(p+1).
\end{equation}
\noindent
This is given in Theorem \ref{thm-quadratic}.

Theorem \ref{bern-zeta1} expresses the rational numbers $B_{n}^{(b)} := 
B_{n}^{(1,b)}$ in terms of the 
values $\zeta_{1,b}^{H}(n)$. This extends the classical result for the 
Bernoulli numbers $B_{n} = B_{n}^{(1,1)}$. Section
\ref{sec-arithmetical} states 
some conjectures on arithmetical properties of the denominators of 
$B_{n}^{(b)}$ extending the von Staudt-Clausen theorem for Bernoulli
numbers. Section \ref{sec-prob} introduces a probabilistic technique to
approach these questions and Section \ref{sec-generalized} discusses a family
of polynomials introduced by K. Dilcher and proposes a natural generalization.

This work provides linear identities linking three types of functions: the 
classical beta function, the hypergeometric Bernoulli polynomials and the 
hypergeometric zeta function. Explicitly, Theorem \ref{thm-linear} gives a
linear recurrence involving the beta function and the hypergeometric zeta 
function, Theorem \ref{thm-binomials} gives a linear recurrence involving the 
beta function (written as binomial coefficients) and the hypergeometric 
Bernoulli polynomials and, finally, Theorem \ref{thm-linear-2} gives a relation
between the hypergeometric Bernoulli numbers and the hypergeometric zeta
function. 

\medskip

\noindent
\textbf{Notation}. It is an unfortunate fact that many of the terms used in 
the present work are denoted by the letter $B$. The list below shows the 
symbols employed here. 

\medskip

\begin{align*}
&B_{n}  & \text{Bernoulli number}   & \quad \eqref{gen-ber} \\
&M(a,b;z) &   \text{Kummer function } &  \quad \eqref{kummer-1-def} \\
&\Phi_{a,b}(z) &   \text{Kummer function } & \quad \eqref{def-kummer} \\
&\zeta_{a,b}^{H}(s) &  \text{hypergeometric zeta function } & 
\quad \eqref{hypergeom-def} \\
&B_{n}^{(b)}  & \text{hypergeometric Bernoulli number } & 
\quad \eqref{def-Phi} \\
&B(a,b) & \text{the beta function} & 
\quad \eqref{beta-def} \\
&\mathfrak{B}_{a,b} &  \text{a beta distributed random variable}
& \quad \eqref{def-beta11} \\
&\mathfrak{Z}_{a,b} & \text{a complex random variable} & 
\quad \eqref{Z-stoch} \\
&B_{n}^{(a,b)}(x) &  \text{hypergeometric Bernoulli polynomial} & 
\quad \eqref{hyp-ber-poly-def} \\
&B_{n}^{(a,b)} &     \text{hypergeometric Bernoulli number} & 
\quad \eqref{hyp-ber-numb-def} 
\end{align*}

\section{Properties of the Kummer function $\Phi_{a,b}(z)$.} 
\label{sec-kummer} 

The function 
\begin{equation}
\Phi_{a,b}(z) = \pFq11{a}{a+b}{z} = M(a,a+b;z), \text{ for } a, \, b  
\in \mathbb{R}
\label{def-kummer1}
\end{equation}
\noindent
defined in terms of the Kummer function $M(a,b;z)$ is the main object 
considered in the present work. The function $M(a,b;z)$ satisfies the 
differential equation
\begin{equation}
z \frac{d^{2}M}{dz^{2}} + (b-z) \frac{dM}{dz} - a M = 0,
\end{equation}
\noindent
obtained from the standard hypergeometric equation
\begin{equation}
z \frac{dw^{2}}{dz^{2}} + \left[ c - (a+b+1)z \right] \frac{dw}{dz} - 
abw = 0
\end{equation}
\noindent
by scaling $z \mapsto z/b$, letting $b \to \infty$ and replacing the 
parameter $c$ by $b$. 

The first result shows that the special 
case $a=1$ gives the function considered by Howard \cite{howard-1967a}.

\begin{theorem}
\label{thm-phi1b}
For $b \in \mathbb{N}$, the function $\Phi_{1,b}(z)$ is given by 
\begin{equation}
\Phi_{1,b}(z) = 
\frac{b!}{z^{b}} \left( e^{z} - \sum_{k=0}^{b-1} \frac{z^{k}}{k!} \right).
\end{equation}
\end{theorem}
\begin{proof}
This follows directly from the expansion
\begin{equation*}
\pFq11{1}{1+b}{z}  =  
\sum_{k=0}^{\infty} \frac{(1)_{k} z^{k}}{(1+b)_{k} \, k!}  
 =  \sum_{k=0}^{\infty} \frac{b!}{(b+k)!} z^{k} \\
 =  \frac{b!}{z^{b}} \sum_{k=b}^{\infty} \frac{z^{k}}{k!}.
\end{equation*}
\end{proof}

\begin{corollary}
The zeta function $\zeta_{HA}(s)$ in \eqref{zeta-has1} is given by 
\begin{equation}
\zeta_{HA}(s) = \frac{b!}{\Gamma(s+b-1)} \int_{0}^{\infty} 
\frac{x^{s-2} \, dx}{\Phi_{1,b}(x)}.
\end{equation}
\end{corollary}

The next property of $\Phi_{a,b}(z)$ is a representation as an 
infinite product. The result comes from the classical Hadamard factorization 
theorem for entire functions. A preliminary lemma is given first. 

\begin{lemma}
\label{der-f11}
The Kummer function $\Phi_{a,b}(z)$ satisfies 
\begin{equation}
\frac{d}{dz} \Phi_{a,b}(z) = 
\frac{a}{a+b} \Phi_{a+1,b}(z).
\end{equation}
\end{lemma}
\begin{proof}
This comes directly from formula 
\begin{equation}
\frac{d}{dz} \pFq11{a}{b}{z} = \frac{a}{b} \pFq11{a+1}{b+1}{z},
\end{equation}
\noindent
which is entry $13.3.15$ on page $325$ of  \cite{lozier-2003a}. 
\end{proof}

The next step is to analyze the factorization of the function 
$\Phi_{a,b}(z)$. Recall that the \textit{order} of an entire 
function $h(z)$ is defined as the infimum of $\alpha \in \mathbb{R}$ for 
which there exists a radius $r_{0} > 0$ such that 
\begin{equation}
|h(z)| < e^{|z|^{\alpha}} \text{ for } |z| > r_{0}.
\end{equation}
\noindent
The order of $h(z)$ is denoted by $\rho(h)$. See \cite{boas-1954a} for 
information on the order of entire functions.  The main result used here 
is Hadamard's theorem stated below.

\begin{theorem}
For $p \in \mathbb{N}$, define the elementary factors 
\begin{equation}
E_{p}(z) = \begin{cases}
    1 - z & \quad \text{ if } p = 0, \\
    (1 - z) \text{exp}\left( z + \frac{z^{2}}{2} + \frac{z^{3}}{3} + 
\cdots + \frac{z^{p}}{p}  \right) 
 & \quad \text{ otherwise}.
          \end{cases}
\end{equation}
\noindent
Assume $h$ is an entire function of finite order $\rho = \rho(h)$. Let 
$\{ a_{n} \}$ be the collection of zeros of $h$ repeated according to 
multiplicity. Then $h$ admits the factorization 
\begin{equation}
h(z) = z^{m} e^{g(z)} \prod_{n=1}^{\infty} E_{p} \left( \frac{z}{a_{n} }
\right).
\end{equation}
\noindent
where $g(z)$ is a polynomial of degree $q \leq \rho, \, 
p = \lf \rho \rf$ and $m \geq 0$ is the order of the zero of $h$ at the origin.
\end{theorem}

The next result establishes the order of $\Phi_{a,b}(z)$. 

\begin{theorem}
Let $a, \, b  > 0$. Then  
$\Phi_{a,b}$ is an entire function 
of order $1$.
\end{theorem}
\begin{proof}
The ratio test shows that the function  $\Phi_{a,b}(z)$ is entire. Moreover 
\begin{equation}
\frac{(a)_{\ell}}{(a+b)_{\ell}} = 
\prod_{k=0}^{\ell-1} \frac{a+k}{a+b+k} < 1,
\end{equation}
\noindent
therefore 
\begin{equation}
|\Phi_{a,b}(z)| \leq 
\sum_{\ell=0}^{\infty} \frac{(a)_{\ell}}{(a+b)_{\ell}} 
\left| \frac{z^{\ell}}{\ell!} \right| \leq 
\sum_{\ell=0}^{\infty} \frac{|z|^{\ell}}{\ell!}  = e^{|z|}.
\end{equation}
\noindent
This proves $\rho(h) \leq 1$. 

To establish the opposite inequality, use the asymptotic behavior 
\begin{equation}
M(a,b;z) \sim \frac{e^{-z} z^{a-b}}{\Gamma(a)}, \text{ as } z \to \infty
\end{equation} 
(see \cite{lozier-2003a}, page $323$) to see that, for every $0 \leq 
\varepsilon < 1$ and $z \in \mathbb{R}$,
\begin{equation}
\lim\limits_{z \to \infty} 
\frac{\Phi_{a,b}(z)}{\text{exp}(z^{\varepsilon})}
= + \infty.
\end{equation}
\noindent 
Hence, for any 
given $r_{0} > 0$, there is $r > r_{0}$ such that 
\begin{equation}
| \Phi_{a,b}(r) | = \Phi_{a,b}(r) > 
\text{exp} \left(r^{\varepsilon} \right) 
= \text{exp}\left(|r|^{\varepsilon} \right). 
\end{equation}
\noindent
This proves $\rho(h) \geq 1$ and the proof is complete.
\end{proof}

The factorization of $\Phi_{a,b}(z)$ in terms of its zeros $\{ z_{a,b;k} \}$
is discussed next. Section $13.9$ of 
\cite{lozier-2003a} states that if $a$ and $b \neq 0, \, -1, \, -2, \, 
\cdots$, then $\Phi_{a,b}(z)$ has infinitely many complex zeros. Moreover, if 
$a, \, b \geq 0$, then there are no real zeros. The growth of the large 
zeros of $M(a,a+b;z)$ is given by 
\begin{equation}
z_{a,b;n} = \pm( 2n+a) \pi \imath + 
\ln \left( - \frac{\Gamma(a)}{\Gamma(b)} ( \pm 2 n \pi \imath)^{b-a} \right)
+ O( n^{-1} \ln n ),
\end{equation}
\noindent
where $n$ is a large positive integer, and the logarithm takes its 
principal value. 
 
\smallskip

\begin{theorem}
\label{thm-product}
The function $\Phi_{a,b}(z)$ admits the factorization 
\begin{equation}
\Phi_{a,b}(z) = e^{az/(a+b)} 
\prod_{k=1}^{\infty} \left( 1 - \frac{z}{z_{a,b;k}} \right) 
e^{z/z_{a,b;k}}.
\end{equation}
\end{theorem}
\begin{proof}
Hadamard's theorem shows the existence of two constants $A, \, B$ such 
that 
\begin{equation}
\Phi_{a,b}(z) = e^{Az+B} \prod_{k=1}^{\infty} 
\left(1 - \frac{z}{z_{a,b;k}} \right)
e^{z/z_{a,b;k}}.
\label{prod-exp}
\end{equation}
\noindent
Evaluating at $z=0$, using $\Phi_{a,b}(0) =1$, gives $B=0$. To obtain 
the value of the parameter $A$, take the logarithmic derivative of 
\eqref{prod-exp} and use Lemma \ref{der-f11} to produce 
\begin{equation}
\frac{a}{a+b} \frac{\pFq11{a+1}{a+b+1}{z}}{\pFq11{a}{a+b}{z}} = 
A + \sum_{k=0}^{\infty} \left[ \frac{1}{z_{a,b;k}} - 
\frac{1}{z_{a,b;k} - z} \right]. 
\end{equation}
\noindent
Expanding both sides near $z=0$ gives 
\begin{equation}
\frac{a}{a+b} + \frac{ab}{(a+b)^{2}(1+a+b)}z + \mathcal{O}(z^{2}) =
A - \sum_{k=1}^{\infty} \frac{z}{z_{a,b;k}^{2}} + 
\mathcal{O}(z^{2}). 
\label{expansion-0}
\end{equation}
\noindent
This gives $A = a/(a+b)$, completing the proof.
\end{proof}

\begin{corollary}
The hypergeometric zeta function has the special value 
\begin{equation}
\zeta_{a,b}^{H}(2) = - \frac{ab}{(a+b)^{2}(1+a+b)}.
\label{value-at2}
\end{equation}
\end{corollary}
\begin{proof}
Compare the coefficients of $z$ in \eqref{expansion-0}.
\end{proof}

\smallskip

The next statement presents additional properties of the Kummer function which
will be useful in the next section. It 
appears as entries $13.4.12$ and $13.4.13$ in \cite{abramowitz-1972a}. 

\begin{lemma}
\label{kummer-56}
The Kummer function satisfies 
\begin{equation}
\frac{d}{dz} \Phi_{a,b}(z) = - \frac{b}{a+b} \Phi_{a,b+1}(z) + 
\Phi_{a,b}(z)
\label{kummer-1}
\end{equation}
\noindent
and 
\begin{equation}
\frac{d}{dz} \Phi_{a,b+1}(z) = \frac{a+b}{z} 
\left( \Phi_{a,b}(z) -  \Phi_{a,b+1}(z) \right).
\label{kummer-2}
\end{equation}
\end{lemma}

\section{Recurrence for the hypergeometric zeta function} 
\label{sec-recurrences} 

This section describes some recurrences for the values $\zeta_{a,b}^{H}(k)$. 
The proofs are based on a relation between the Kummer function 
$\Phi_{a,b}(z)$ and these values. It is the analog of standard result for the 
usual zeta function
\begin{equation}
\sum_{k=1}^{\infty} \zeta(k+1)z^{k} = - \gamma - \psi(1-z)
\end{equation}
\noindent 
where $\psi(z) = \frac{d}{dz} \log \Gamma(z)$ is the digamma function and 
$\gamma$ is the Euler constant. This 
relation is obtained directly from the product representation of $\Gamma(z)$. 
See entry $6.3.14$ in \cite{abramowitz-1972a}.

\begin{prop}
\label{kummer-90}
The Kummer function $\Phi_{a,b}(z)$ and the hypergeometric 
zeta function are related by 
\begin{equation}
\frac{\Phi_{a,b+1}(z)}{\Phi_{a,b}(z)} = 1 + \frac{a+b}{b} 
\sum_{k=1}^{\infty} \zeta_{a,b}^{H}(k+1)z^{k}.
\end{equation}
\end{prop}
\begin{proof}
The relation \eqref{kummer-1} gives 
\begin{equation}
\frac{\Phi_{a,b+1}(z)}{\Phi_{a,b}(z)} = -\frac{a+b}{b} 
\left[ \frac{\Phi_{a,b}'(z)}{\Phi_{a,b}(z)} - 1 \right].
\end{equation}
\noindent
The fraction on the right-hand side is the logarithmic derivative of 
the product in Theorem \ref{thm-product}. This yields
\begin{equation}
\label{kummer-3}
\frac{\Phi_{a,b+1}(z)}{\Phi_{a,b}(z)} = -\frac{a+b}{b} 
\left(\frac{a}{a+b} + \sum_{k=0}^{\infty} 
\left[ \frac{1}{z_{a,b;k}} - \frac{1}{z_{a,b;k}-z} \right] - 1 
\right).
\end{equation}
\noindent 
To establish the result, use the expansion
\begin{equation}
\frac{1}{z_{a,b;k}-z} = \frac{1}{z_{a,b;k}} \frac{1}{1 - z/z_{a,b;k}} 
\end{equation}
\noindent
and expand the last term as the sum of a geometric series. Since 
$\min_{k} \left\{ |z_{a,b;k} | \right\} \neq 0$, this series 
has a positive radius of convergence.
\end{proof}

The next result gives a linear recurrence for the hypergeometric 
zeta function. This involves the beta function 
\begin{equation}
B(u,v) = \int_{0}^{1} x^{u-1}(1-x)^{v-1} \, dx = \frac{\Gamma(u) \Gamma(v)}
{\Gamma(u+v)},
\label{beta-def}
\end{equation}
\noindent
with values
\begin{equation}
B(u,v) = \frac{u+v}{uv} \binom{u+v}{u}^{-1} 
\text{ for }u, \, v \in \mathbb{N}.
\label{form-binom}
\end{equation}

\smallskip

\begin{theorem}
\label{thm-linear}
The hypergeometric zeta function satisfies the linear recurrence
\begin{equation}
\sum_{\ell=1}^{p} B(a+p-\ell,b) \frac{p!}{(p- \ell)!} 
\zeta_{a,b}^{H}(\ell+1) = 
- \frac{bp}{(a+b)(a+b+p)}B(a+p,b).
\label{linear-rec}
\end{equation}
\end{theorem}
\begin{proof}
Proposition \ref{kummer-90} gives 
\begin{equation}
\Phi_{a,b+1}(z) = \Phi_{a,b}(z) + \frac{a+b}{b} \Phi_{a,b}(z) 
\sum_{k=1}^{\infty} \zeta_{a,b}^{H}(k+1) z^{k}
\end{equation}
\noindent
Matching the coefficient of $z^{p}$ gives the identity
\begin{equation}
\frac{(a)_{p}}{(a+b+1)_{p}} = \frac{(a)_{p}}{(a+b)_{p}} + 
\frac{a+b}{b} \sum_{\ell=1}^{p} \frac{(a)_{p- \ell}}{(a+b)_{p- \ell}} 
\frac{p!}{(p - \ell)!} \, \zeta_{a,b}^{H}(\ell+1).
\end{equation}
\noindent
This simplifies to produce the result.
\end{proof}

\begin{note}
The recurrence \eqref{linear-rec} is written as 
\begin{equation*}
\zeta_{a,b}^{H}(p) = - \frac{bB(a+p-1,b)}{(a+b)(a+b+p-1)(p-2)!B(a,b)} 
-\sum_{r=1}^{p-2} 
\frac{B(a+r,b)}{B(a,b)r!} \zeta_{a,b}^{H}(p-r).
\end{equation*}
\noindent
The initial condition given in \eqref{value-at2}
shows that, for $p \in \mathbb{N}$, the value 
$\zeta_{a,b}^{H}(p)$ is a rational function of $a, \, b$. The first few values
are 
\begin{eqnarray*}
\zeta_{a,b}^{H}(2) & = &  - \frac{ab}{(a+b)^{2}(1+a+b)} \\
\zeta_{a,b}^{H}(3) & = &  \frac{ab(a-b)}{(a+b)^{3} (a+b+1)(a+b+2)} \\
\zeta_{a,b}^{H}(4) & = &  - \frac{ab P_{4}(a,b)}{(a+b)^{4}(a+b+1)^{2} 
(a+b+2)(a+b+3)}
\end{eqnarray*}
\noindent
where 
\begin{equation*}
P_{4}(a,b) = a^{2} + a^{3} -4ab -2a^{2}b + b^{2} 
- 2ab^{2} + b^{3}.
\end{equation*}
\noindent
The authors were unable to discern the patterns in $\zeta_{a,b}^{H}(p)$.
\end{note}

The next result presents a different type of recurrence for 
$\zeta_{a,b}^{H}(s)$.  It is the analogue of the classical identity 
\begin{equation}
\left( n + \tfrac{1}{2} \right) \zeta(2n) = 
\sum_{k=1}^{n-1} \zeta(2k) \zeta(2n-2k), \text{ for } n \geq 2.
\end{equation}
\noindent 
See $25.6.16$ in \cite{lozier-2003a}. 

\begin{theorem}
\label{thm-quadratic}
The hypergeometric zeta function satisfies the quadratic recurrence
\begin{equation}
\sum_{k=1}^{p} \zeta_{a,b}^{H}(k+1)\zeta_{a,b}^{H}(p-k+1) = 
(a+b+p+1) \zeta_{a,b}^{H}(p+2) + \left( \frac{a-b}{a+b} \right)
 \zeta_{a,b}^{H}(p+1).
\end{equation}
\end{theorem}
\begin{proof}
The function $f(z) = \Phi_{a,b+1}(z)/\Phi_{a,b}(z)$ satisfies the differential 
equation
\begin{equation}
f'(z) = \frac{a+b}{z} (1 - f(z)) - f(z) + \frac{b}{a+b}f^{2}(z).
\end{equation}
\noindent
This can be verified directly using the results of Lemma \ref{kummer-56}. Now 
use Theorem \ref{kummer-90} to match the coefficients of $z^{p}$ and 
produce 
\begin{eqnarray}
\frac{a+b}{b} (p+1) \zeta_{a,b}^{H}(p+2) & = & 
- \frac{(a+b)^{2}}{b} \zeta_{a,b}^{H}(p+2) - \frac{a+b}{b} \zeta_{a,b}^{H}(p+1) 
\nonumber \\
&  & + \frac{a+b}{b} \sum_{k=1}^{p} \zeta_{a,b}^{H}(k+1) \zeta_{a,b}^{H}(p-k+1)
 + 2 \zeta_{a,b}^{H}(p+1),
\nonumber
\end{eqnarray}
which, after simplification, yields the result.
\end{proof}

\begin{note}
Matching the constant terms recovers the 
value of $\zeta_{a,b}^{H}(2)$ in \eqref{value-at2}.
\end{note}

\section{The hypergeometric Bernoulli numbers} 
\label{sec-hyperB} 

This section considers properties of the \textit{hypergeometric Bernoulli
numbers} $B_{n}^{(b)}$, defined by the relation
\begin{equation}
\frac{1}{\Phi_{1,b}(z)} = \sum_{n=0}^{\infty} 
B_{n}^{(b)} \frac{z^{n}}{n!}.
\label{howard-1}
\end{equation}
\noindent
These are precisely the numbers $A_{b,n}$ studied by Howard 
\cite{howard-1967a}. This follows from Theorem \ref{thm-phi1b} and 
\eqref{howard-gen}. The special case 
$b=1$ corresponds to the Bernoulli numbers.

\medskip

The next result appears in \cite{howard-1967a} in the case $b=2$.

\begin{theorem}
\label{bern-zeta1}
Let $b \in \mathbb{N}$. The hypergeometric Bernoulli numbers $B_{n}^{(b)}$ 
are expressed in terms of the hypergeometric zeta function as 
\begin{equation}
\label{bern-terms-zeta}
B_{n}^{(b)} = \begin{cases}
          1 & \quad \text{ for } n = 0 \\
         -1/(1+b) & \quad \text{ for } n = 1 \\
          - n! \zeta_{1,b}^{H}(n)/b  & \quad \text{ for } n \geq 2.
            \end{cases}
\end{equation}
\end{theorem}
\begin{proof}
The product representation of $\Phi_{1,b}(z)$ given in Theorem 
\ref{thm-product} is 
\begin{equation}
\Phi_{1,b}(z) = e^{z/(1+b)} \prod_{k=1}^{\infty} \left( 1 - \frac{z}{z_{1,b;k}}
\right)e^{z/z_{1,b;k}}. 
\end{equation}
\noindent
Logarithmic differentiation yields
\begin{eqnarray*}
\frac{\Phi'_{1,b}(z)}{\Phi_{1,b}(z)} & = & 
\frac{1}{1+b} + \sum_{k=1}^{\infty} \left[ \frac{1}{z_{1,b;k}} - 
\frac{1}{z_{1,b;k}-z} \right] \\
& = & \frac{1}{1+b} + \sum_{k=1}^{\infty} \frac{1}{z_{1,b;k}} 
\left( 1 - \sum_{\ell=0}^{\infty} \left( \frac{z}{z_{1,b;k}} \right)^{\ell} 
\right) \\
& = & \frac{1}{1+b} - \sum_{\ell=1}^{\infty} z^{\ell} \zeta_{1,b}^{H}(\ell+1).
\end{eqnarray*}
\noindent
On the other hand, Theorem \ref{thm-phi1b} gives 
\begin{equation}
\Phi_{1,b}(z) = \frac{b!}{z^{b}} \left( e^{z} - \sum_{j=0}^{b-1} 
\frac{z^{j}}{j!} \right)
\end{equation}
\noindent
and logarithmic differentiation produces 
\begin{equation}
\frac{\Phi'_{1,b}(z)}{\Phi_{1,b}(z)} = 1 + \frac{b}{z} \cdot 
\frac{1}{\Phi_{1,b}(z)} - \frac{b}{z}.
\end{equation}
\noindent
Therefore,
\begin{equation}
\frac{1}{\Phi_{1,b}(z)} = 1 - \frac{z}{1+b} - \frac{1}{b} 
\sum_{\ell=2}^{\infty} z^{\ell} \zeta_{1,b}^{H}(\ell) = 
\sum_{n=0}^{\infty} B_{n}^{(b)} \frac{z^{n}}{n!}.
\end{equation}
\noindent
The conclusion follows by comparing coefficients of powers of $z$. 
\end{proof}

\begin{note}
The previous theorem suggests the definition 
\begin{equation}
\zeta_{1,b}^{H}(1) = \frac{b}{1+b}. 
\end{equation}
\end{note}

\section{Some arithmetical conjectures} 
\label{sec-arithmetical} 

Theorem \ref{thm-linear} and the relation \eqref{bern-terms-zeta} show that 
the coefficients $B_{n}^{(b)}$ are rational numbers. For example:
\begin{eqnarray*}
B_{n}^{(1)} & = &  
\{ 1, \, - \tfrac{1}{2}, \, \tfrac{1}{6}, \, 0, \, -\tfrac{1}{30}
, \, 0, \, \tfrac{1}{42}, \, 0, \,  \cdots \}, \\
B_{n}^{(2)} & = &  
\{ 1, \, - \tfrac{1}{3}, \, \tfrac{1}{18}, \, \tfrac{1}{90}, \, -\tfrac{1}{270}
, \, - \tfrac{5}{1134}, \, -\tfrac{1}{5670}, \, \tfrac{7}{2430}, \cdots \}, \\
B_{n}^{(3)} & = &  
\{ 1, \, - \tfrac{1}{4}, \, \tfrac{1}{40}, \, \tfrac{1}{160}, \, \tfrac{1}{5600}
, \, - \tfrac{1}{896}, \, -\tfrac{13}{19200}, \, \tfrac{7}{76800}, \cdots \}.
\end{eqnarray*}

The arithmetic properties of the Bernoulli numbers $B_{n} = B_{n}^{(1)}$ 
are very intriguing. It is well-known that the
Bernoulli numbers of odd index vanish (except $B_{1} = -1/2$), leaving
$B_{2n}$ for consideration. Write $B_{2n} = N_{2n}/D_{2n}$ in reduced form.
The arithmetic properties of $D_{2n}$ include the 
von Staudt-Clausen theorem stated below. The numerators $N_{2n}$ are much 
more difficult to analyze. Chapter 15 of \cite{ireland-1990a} contains 
information about their relation to Wiles theorem (previously known as 
Fermat's last theorem). 

\begin{theorem}
The denominator of the Bernoulli number $B_{2n}$ is given by 
\begin{equation}
D_{2n} = \prod_{(p-1)|2n} p.
\end{equation}
\end{theorem}

In particular, the denominator of $B_{2n}$ is even (actually always divisible
by $6$) and it is square-free. In the case of the hypergeometric Bernoulli 
numbers, computer experiments suggest an extension of these properties. Let 
\begin{equation}
\mathfrak{D}(b) = \left\{ \text{denominator}\left(B_{n}^{(b)} \right): \, 
n \geq 0 \right\}.
\end{equation}
\noindent
The examples 
\begin{eqnarray}
\mathfrak{D}(2) & = &  \left\{ 1, \, 3, \, 18, \, 90, \, 270, \, 1134, \, 5670, \, 
2430, \cdots \right\} \\
\mathfrak{D}(3) & = & 
\left\{ 1, \, 4, \, 40, \, 160, \, 5600, \, 896, \, 19200, \, 
76800, \cdots \right\}  \nonumber \\
\mathfrak{D}(4) & = &  \left\{ 1, \, 5, \, 75, \, 875, \, 26250, \, 78750, \, 
918750, \, 
3093750, \cdots \right\}, \nonumber
\end{eqnarray}
\noindent
show that $\mathfrak{D}(b)$ contains an initial segment of \textit{odd} 
numbers. 

\begin{conjecture}
Let $\alpha(b)$ be the number of odd terms at the beginning of 
$\mathfrak{D}(b)$. Then
\begin{equation}
\alpha(b) = \begin{cases}
\nu_{2}(b) + 1 & \quad \text{ if }b \not \equiv 0 \bmod 4 \\
2^{\nu_{2}(b)} & \quad \text{ if }b \equiv 0 \bmod 4. 
\end{cases}
\end{equation}
\end{conjecture}

The prime factorization of 
$D_{2n} = \text{den}(B_{2n})$, shows 
that if $p$ is a prime dividing $D_{2n}$, then $p \leq 2n+1$. 
Numerical evidence of the corresponding statement for 
$B_{n}^{(b)}$ leads to the next conjecture. 

\begin{conjecture}
Every prime $p$ dividing the denominator of $B_{n}^{(b)}$ satisfies 
$p \leq n+b$.
\end{conjecture}

These two conjectures have been verified up to $b = 1000$. 

\begin{note}
It is no longer true that the denominators are square-free. For 
example $B_{7}^{(3)} = 7/76800$ and $76800 = 2^{10} \cdot 3 \cdot 5^{2}$. 
\end{note}

\section{A probabilistic approach} 
\label{sec-prob} 

This section presents an interpretation of the Kummer function $\Phi_{a,b}(z)$
as the expectation of a complex random variable. For a random variable 
$\mathfrak{R}$ with a continuous distribution function $r(x)$, the expectation
operator is defined by 
\begin{equation}
\mathbb{E}u(\mathfrak{R}) = \int_{\mathbb{R}} u(x) r(x) \, dx,
\end{equation}
\noindent
for the class of functions $u$ for which the integral is finite. 

The techniques employed here 
were recently used in 
\cite{amdeberhan-2013e} to solve a problem proposed by D. Zeilberger related 
to the Narayana polynomials discussed in \cite{lasallem-2012a}.

\begin{definition}
\label{def-beta}
Let $a, \, b > 0$. The 
random variable $\mathfrak{B}_{a,b}$ is called \textit{beta distributed}, 
or simply \textit{a beta random variable}, if 
its distribution function is given by 
\begin{equation}
\label{def-beta11}
f_{\mathfrak{B}_{a,b}}(x) = \begin{cases}
    \frac{1}{B(a,b)} x^{a-1}(1-x)^{b-1} & \quad \text{ if } 0 \leq x \leq 1 \\
      0 & \quad \text{ otherwise.}
\end{cases}
\end{equation}
\noindent 
Here $B(a,b)$ is the beta function \eqref{beta-def}.
\end{definition}

The integral representation \cite[13.2.1]{abramowitz-1972a} 
\begin{equation}
\Phi_{a,b}(z) = \pFq11{a}{a+b}{z} = \frac{1}{B(a,b)} 
\int_{0}^{1} e^{tz} t^{a-1}(1-t)^{b-1} \, dt,
\end{equation}
\noindent
shows that $\Phi_{a,b}(z)$ is the moment generating function of a beta 
random variable $\mathfrak{B}_{a,b}$:
\begin{equation}
\mathbb{E} \left( e^{z \mathfrak{B}_{a,b}} \right) = \Phi_{a,b}(z),
\end{equation}
\noindent
where $\mathbb{E}$ is the expectation operator. 

\smallskip

This representation is used to construct a new random variable. Some 
preliminary discussion is given first. 

\smallskip 

A random variable $\Gamma$ is said to be \textit{exponentially distributed} if 
its distribution function is given by 
\begin{equation}
\label{exp-distr}
f_{\Gamma}(x) = \begin{cases} 
e^{-x} & \quad \text{ if }  x \geq  0, \\
0 & \quad \text{ elsewhere.}
\end{cases}
\end{equation}
\noindent
The moment generating function of an exponentially distributed random 
variable $\Gamma$ is 
\begin{equation}
\mathbb{E} \left[ e^{z \Gamma} \right]  = \frac{1}{1-z}, \text{ for } |z| < 1.
\end{equation}

\begin{note}
The authors hope that not too much confusion will be created by using 
the symbol $\Gamma$ for this kind of random variables. The choice of 
name is clear: if $\Gamma$ is exponentially distributed, then 
\begin{equation}
\mathbb{E} \left[ \Gamma^{\alpha-1} \right] = \Gamma(\alpha).
\end{equation}
\end{note}

Consider a sequence $\{\Gamma_{k} \}_{k \geq 1}$ of independent 
identically distributed 
random variables, each with the same exponential distribution 
\eqref{exp-distr}. 

\begin{definition}
Let $\{z_{a,b;k}: \, k \in \mathbb{N} \}$ be the collection of zeros of the 
Kummer function $\Phi_{a,b}(z)$. The complex-vaued random variable 
$\mathfrak{Z}_{a,b}$ is defined by 
\begin{equation}
\mathfrak{Z}_{a,b} = 
- \frac{a}{a+b} + \sum_{k=1}^{\infty} \frac{\Gamma_{k}-1}{z_{a,b;k}}.
\label{Z-stoch}
\end{equation}
\end{definition}

Some properties of $\mathfrak{Z}_{a,b}$ are given below. 
The main relation between $\mathfrak{Z}_{a,b}$ and the Kummer function 
$\Phi_{a,b}(z)$ is stated first. 

\begin{theorem}
\label{Z-def}
The complex-valued random variable $\mathfrak{Z}_{a,b}$ satisfies
\begin{equation}
\mathbb{E} e^{z \mathfrak{Z}_{a,b}} = \frac{1}{\Phi_{a,b}(z)}.
\end{equation}
\end{theorem}
\begin{proof}
The independence of the family $\{ \Gamma_{k} \}$ and 
Theorem \ref{thm-product} give 
\begin{equation}
\mathbb{E} e^{z \mathfrak{Z}_{a,b}} = 
e^{-az/(a+b)} \prod_{k=1}^{\infty} \frac{1}{1 - z/z_{a,b;k}} e^{-z/z_{a,b;k}}.
\end{equation}
\noindent
This is the stated result.
\end{proof}

\begin{lemma}
The symmetry property 
\begin{equation}
\mathfrak{Z}_{a,b} =  -1 - \mathfrak{Z}_{b,a},
\label{symmetry-z}
\end{equation}
\noindent
holds in the sense of distribution.
\end{lemma}
\begin{proof}
The symmetry property follows from Kummer transformation 
\begin{equation}
\pFq11{a}{b}{z} = e^{z} \pFq11{b-a}{b}{-z}.
\label{kummer-trans}
\end{equation}
\end{proof}

Now consider $\mathfrak{B}_{a,b}$ to be a beta-distributed random variable 
independent of $\mathfrak{Z}_{a,b}$. Their 
moment generating functions satisfy
\begin{equation}
1 = \mathbb{E} \left[ e^{z \mathfrak{B}_{a,b}} \right] \times 
\mathbb{E} \left[ e^{z \mathfrak{Z}_{a,b}} \right] = 
\mathbb{E} \left[ e^{z ( \mathfrak{B}_{a,b}+ \mathfrak{Z}_{a,b})} \right].
\end{equation}
\noindent
Expanding in a power series yields the identity
\begin{equation}
\mathbb{E} \left[ ( \mathfrak{B}_{a,b} + \mathfrak{Z}_{a,b} )^{n} \right] = 
\delta_{n} = 
\begin{cases} 
1 & \quad \text{ if } n = 0, \\
0 & \quad \text{ if } n \neq 0.
\end{cases}
\label{expectation-1}
\end{equation}
\noindent
This is now used 
to provide a probabilistic representation of an 
analytic function.

\begin{theorem}
\label{thm-taylor}
Let $\mathfrak{B}_{a,b}$ and $\mathfrak{Z}_{a,b}$ as before. Then, if $f$ is 
an analytic function,
\begin{equation}
\mathbb{E} f\left(z+ \mathfrak{B}_{a,b} + \mathfrak{Z}_{a,b} \right) = f(z),
\end{equation}
\noindent
provided the expectation is finite.
\end{theorem}
\begin{proof}
It suffices to establish the result for $f(z) = z^{p}$, with $p \in 
\mathbb{N}$. This follows directly from \eqref{expectation-1}:
\begin{eqnarray*}
\mathbb{E} \left[ ( z + \mathfrak{B}_{a,b} + \mathfrak{Z}_{a,b} )^{p} \right] 
& = &  \sum_{k=0}^{p} \binom{p}{k} z^{p-k} \mathbb{E} \left[ 
\left(\mathfrak{B}_{a,b} + \mathfrak{Z}_{a,b} \right)^{k} \right] \\
 & = & z^{p}.
\end{eqnarray*}
\end{proof}

\begin{note}
Hassen and Nguyen \cite{hassen-2008a} show that 
\begin{equation}
\int_{0}^{1} x^{a-1}(1-x)^{b-1} B_{n}^{(a,b)}(x) \, dx = 
\begin{cases}
B(a,b) & \quad \text{ if } n = 0 \\
0 & \quad \text{ if } n \neq  0.
\end{cases}
\end{equation}
\noindent
This follows directly from Theorem \ref{thm-taylor} by taking 
$f(z) = z^{n}$ and then $z=0$.
\end{note}

The next item in this section gives a probabilistic point of view of the 
linear recurrence in Theorem \ref{thm-linear}. Some 
preliminary background is discussed first. 

\begin{definition}
\label{def-conj}
The independent random variables $X$ and $Y$ are called \textit{conjugate} if 
\begin{equation}
\mathbb{E} \left[(X+Y)^{n} \right] = \delta_{n}.
\end{equation}
\end{definition}

For example, \eqref{expectation-1} shows that 
 $\mathfrak{B}_{a,b}$ and $\mathfrak{Z}_{a,b}$ are conjugate.

\begin{definition}
Let $X$ be a random variable. The \textit{moment generating function} of $X$ is 
\begin{equation}
\varphi_{X}(t) = \sum_{n=0}^{\infty} \mathbb{E} \left[ X^{n} \right] 
\frac{t^{n}}{n!}.
\end{equation}
The sequence of \textit{cumulants} $\kappa_{X}(n)$ is defined by
\begin{equation}
\log \varphi_{X}(z) = \sum_{m=1}^{\infty} \kappa_{X}(m) \frac{z^{m}}{m!}.
\end{equation}
\end{definition}

\begin{example}
\label{cumulants-beta}
The cumulant generating function for the $\mathfrak{B}_{a,b}$ distribution is
\begin{equation}
\log \Phi_{a,b}(z) = \frac{a}{a+b}z - \sum_{p=2}^{\infty} \frac{z^{p}}{p}
\zeta_{a,b}^{H}(p),
\end{equation}
\noindent
so the cumulants are 
\begin{equation}
\kappa_{\mathfrak{B}_{a,b}}(1) = \frac{a}{a+b} \text{ and }
\kappa_{\mathfrak{B}_{a,b}}(p) = - (p-1)! \zeta_{a,b}^{H}(p) \text{ for } 
p \geq 2.
\end{equation}
\end{example}

For a general random variable $X$, the moments $\mathbb{E} X^{n}$ and its 
cumulants $\kappa_{X}(n)$ are related by 
\begin{equation}
\kappa_{X}(n) = \mathbb{E} X^{n} - \sum_{j=1}^{n-1} \binom{n-1}{j-1} 
\kappa_{X}(j) \mathbb{E} X^{n-j}.
\label{mom-cum}
\end{equation}
\noindent
See \cite{smith-1995a} for details. In the special case of a random variable 
with a beta distribution, this gives 
\begin{equation}
(n-1)! \sum_{j=2}^{n} \frac{B(a+n-j,b)}{(n-j)!} \zeta_{a,b}^{H}(j) = 
\frac{a}{a+b} B(a+n-1,b) - B(a+n,b)
\end{equation}
\noindent
that is equivalent to the linear recurrence identity in 
Theorem \ref{thm-linear}.

\section{The generalized Bernoulli polynomials} 
\label{sec-generalized} 

K. Dilcher introduced in \cite{dilcher-2002a} the generalized Bernoulli 
polynomials by 
\begin{equation}
\frac{e^{xz}}{\Phi_{1,b}(z)} = \sum_{k=0}^{\infty} B_{k}^{(b)}(x) 
\frac{z^{k}}{k!}.
\end{equation}

These polynomials are now interpreted as moments.

\begin{theorem}
\label{dilcher-prob}
The generalized Bernoulli polynomials are given by 
\begin{equation}
B_{k}^{(b)}(x) = \mathbb{E}( x + \mathfrak{Z}_{1,b})^{k}.
\end{equation}
\end{theorem}
\begin{proof}
This follows directly from Theorem \ref{Z-def}.
\end{proof}

Dilcher \cite{dilcher-2002a} used generating functions to provide the 
following recursion for these polynomials.

\begin{theorem}[Dilcher]
\label{dilcher-1}
The generalized Bernoulli polynomials $B_{k}^{(b)}(x)$ satisfy
\begin{equation}
B_{k}^{(b)}(x+1) = \sum_{p=0}^{b-1} \binom{k}{p} B_{k-p}^{(b)}(x) + 
\binom{k}{b}x^{k-b}.
\end{equation}
\end{theorem}

A probabilistic proof is presented next. A preliminary result is stated 
first.

\begin{lemma}
Let $\mathfrak{B}_{a,b}$ be a random variable with  a beta distribution
and $g \in C^{1}[0,1]$. Then, for $a, b > 1$,
\begin{equation}
\mathbb{E} g'(x + \mathfrak{B}_{a,b}) = (a+b-1) 
\left[ \mathbb{E} g(x+ \mathfrak{B}_{a,b-1}) - \mathbb{E} g(x + 
\mathfrak{B}_{a-1,b}) \right].
\end{equation}
\noindent
For $a=1$ and $b > 1$
\begin{equation}
\mathbb{E} g'(x+ \mathfrak{B}_{1,b}) = -b g(x) + b \, \mathbb{E}g(x+ 
\mathfrak{B}_{1,b-1}).
\label{for-b=1}
\end{equation}
\end{lemma}
\begin{proof}
A direct calculation shows 
\begin{eqnarray*}
B(a,b) \mathbb{E} g'(x + \mathfrak{B}_{a,b})   & = & 
\int_{0}^{1} g'(x+ t) t^{a-1}(1-t)^{b-1} dt \\
 & = &  g(x+t)t^{a-1}(1-t)^{b-1}\Big{|}_{0}^{1}   \\
& & - (a-1) \int_{0}^{1} g(x+t)t^{a-2}(1-t)^{b-1} \, dt \\
& & + (b-1) \int_{0}^{1} g(x+t)t^{a-1}(1-t)^{b-2} \, dt. 
\end{eqnarray*}
\noindent
The result follows by simplification. The case $a=1$ is straightforward.
\end{proof}

The formula \eqref{for-b=1} can be 
extended directly to higher order derivatives. 

\begin{lemma}
\label{exp-3}
Let $g \in C^{k}[0,1]$ and $\mathfrak{B}_{1,b}$ as before. Then, provided
$b \geq k$, 
\begin{equation}
\mathbb{E} g^{(k)}(x + \mathfrak{B}_{1,b}) = 
- \sum_{\ell=1}^{k} \frac{b!}{(b - \ell)!} g^{(k-\ell)}(x) 
+ \frac{b!}{(b-k)!} \mathbb{E} g(x + \mathfrak{B}_{1,b-k}).
\end{equation}
\end{lemma}

To prove Dilcher's theorem, replace $x$ by $x + \mathfrak{B}_{1,b}$ and 
take the expectation with respect to $\mathfrak{B}_{1,b}$ to see that 
Theorem \ref{dilcher-1} is equivalent to the identity 
\begin{equation}
(x+1)^{k} = \sum_{p=0}^{b-1} \binom{k}{p}x^{k-p} + 
\mathbb{E} ( x + \mathfrak{B}_{1,b-k})^{k}.
\end{equation}
\noindent
This is precisely the statement of Lemma \ref{exp-3} for $g(x) = x^{k}$. 

\smallskip

The expression for the polynomials $B_{k}^{(b)}(x)$ given in Theorem 
\ref{dilcher-prob} provides a natural way to extend them to a two-parameter 
family. 

\begin{definition}
The \textit{hypergeometric Bernoulli polynomials} are defined by 
\begin{equation}
B_{n}^{(a,b)}(x) = \mathbb{E} ( x + \mathfrak{Z}_{a,b})^{n}
\label{hyp-ber-poly-def}
\end{equation}
\noindent
and the \textit{hypergeometric Bernoulli numbers} by
\begin{equation}
B_{n}^{(a,b)} = B_{n}^{(a,b)}(0) = \mathbb{E} (\mathfrak{Z}_{a,b})^{n}.
\label{hyp-ber-numb-def}
\end{equation}
\end{definition}

\begin{note}
The case considered by Howard is $B_{n}^{(b)} = B_{n}^{(1,b)}$.
\end{note}

\begin{prop}
The exponential generating function for the polynomials 
$B_{n}^{(a,b)}(x)$ is given by
\begin{equation}
\sum_{n =0}^{\infty} B_{n}^{(a,b)}(x) \frac{z^{n}}{n!} = 
\frac{e^{xz}}{\Phi_{a,b}(z)}.
\end{equation}
\end{prop}

The next result appears, in the special case $x=0$, as Proposition $2.1$ 
in \cite{dilcher-2002a}. The result gives a change of basis formula from 
$\left\{ B_{n}^{(a,b)}(x): \, n = 0, \, 1, \, 2, \, \cdots 
 \right\}$ to $\left\{ x^{n}: \, n = 0, \, 1, \, 2, \, \cdots \right\}$. 

\begin{theorem}
\label{thm-binomials}
The polynomial $B_{n}^{(a,b)}(x)$ satisfy $B_{0}^{(a,b)}(x) = 1$ and, for 
$n \geq 1$, 
\begin{equation}
\sum_{k=0}^{n} \binom{a+b+n-1}{k} \binom{a-1+n-k}{a-1} B_{k}^{(a,b)}(x) = 
(a+b)_{n} \frac{x^{n}}{n!}.
\end{equation}
\end{theorem}
\begin{proof}
Theorem \ref{thm-taylor}, with $f(x) = x^{n}$ gives 
\begin{equation}
\mathbb{E} ( x + \mathfrak{B}_{a,b} + \mathfrak{Z}_{a,b})^{n} = x^{n}.
\end{equation}
\noindent
The binomial theorem now gives 
\begin{equation}
\sum_{k=0}^{n} \binom{n}{k} \mathbb{E} ( x + \mathfrak{Z}_{a,b} )^{k}  
\mathbb{E} \mathfrak{B}_{a,b}^{n-k} = x^{n}.
\label{binomial-11}
\end{equation}
\noindent
The moments of the beta random variable $\mathfrak{B}_{a,b}$ are 
\begin{eqnarray*}
\mathbb{E} \mathfrak{B}_{a,b}^{p} & = &  \frac{1}{B(a,b)} 
\int_{0}^{1} x^{p} \cdot x^{a-1}(1-x)^{b-1} \, dx  \\
 & = & \frac{B(a+p,b)}{B(a,b)} = 
\frac{\Gamma(a+p)}{\Gamma(a+b+p)} \cdot \frac{\Gamma(a+b)}{\Gamma(a)}. 
\end{eqnarray*}
\noindent
The expression \eqref{binomial-11} is now expressed as 
\begin{equation}
\frac{\Gamma(a+b)}{\Gamma(a)} \sum_{k=0}^{n} \binom{n}{k} 
\frac{\Gamma(a+n-k)}{\Gamma(a+b+n-k)} B_{k}^{(a,b)}(x) = x^{n},
\end{equation}
\noindent
and this is equivalent to the stated result.
\end{proof}

The probabilistic approach presented here, provides a direct proof of a
symmetry property established in \cite{dilcher-2002a}. It 
extends the classical relation $B_{n}(1-x) = (-1)^{n} B_{n}(x)$ of the 
Bernoulli polynomials.

\begin{theorem}
The polynomials $B_{n}^{(a,b)}(x)$ satisfy the symmetry 
\begin{equation}
B_{n}^{(a,b)}(1-x) = (-1)^{n} B_{n}^{(b,a)}(x).
\end{equation}
\end{theorem}
\begin{proof}
The moment representation 
\begin{equation}
B_{n}^{(a,b)}(x) = \mathbb{E}(x + \mathfrak{Z}_{a,b})^{n}
\end{equation}
\noindent
and using \eqref{symmetry-z} yields
\begin{eqnarray*}
B_{n}^{(a,b)}(1-x) & = & \mathbb{E}( 1 - x + \mathfrak{Z}_{a,b})^{n} \\
 & = & \mathbb{E}( - x - \mathfrak{Z}_{b,a})^{n} \\
 & = & (-1)^{n} \mathbb{E} ( x + \mathfrak{Z}_{b,a})^{n}.
\end{eqnarray*}
\end{proof}

The next result presents a linear recurrence for the polynomials 
$B_{n}^{(a,b)}(x)$.

\begin{theorem}
\label{thm-conjugate}
Let $X$ and $Y$ be conjugate random variables. Define the polynomials
\begin{equation}
P_{n}(z) = \mathbb{E}(z+X)^{n} \text{ and }
Q_{n}(z) = \mathbb{E}(z+Y)^{n}.
\end{equation}
\noindent
Then $P_{n}$ and $Q_{n}$ satisfy the recurrences 
\begin{equation}
P_{n+1}(z) - zP_{n}(z) = \sum_{j=0}^{n} \binom{n}{j} \kappa_{X}(j+1)
P_{n-j}(z) 
\end{equation}
\noindent
and 
\begin{equation}
Q_{n+1}(z) - zQ_{n}(z) = - \sum_{j=0}^{n} \binom{n}{j} \kappa_{X}(j+1)
Q_{n-j}(z).
\end{equation}
\end{theorem}
\begin{proof}
Let $X_{1}$ and $X_{2}$ be two independent random variables distributed
as $X$ and let
\begin{eqnarray*}
f(z) & = & 
\mathbb{E} \left[ X_{1}(X_{1}+ Y + z + X_{2})^{n} - X_{1}(z + X_{2})^{n}
\right] \\
& = & \sum_{j=0}^{n} \binom{n}{j} \mathbb{E} \left[ X_{1}(X_{1}+Y)^{j} 
(z+ X_{2})^{n-j} \right] - \mathbb{E} X_{1} (z + X_{2})^{n} \\
& = & \sum_{j=1}^{n} \binom{n}{j} \mathbb{E} \left[ X_{1}(X_{1}+Y)^{j} 
\right] \, \mathbb{E} ( z + X_{2})^{n-j}.
\end{eqnarray*}
\noindent
Theorem $3.3$ in \cite{dinardo-2010a} shows that the cumulants satisfy 
\begin{equation}
\kappa_{X}(p) = \mathbb{E} \left[X(X+Y)^{p-1}\right], \text{ for } p \geq 1.
\end{equation}
\noindent
Therefore 
\begin{equation}
f(z) = \sum_{j=1}^{n} \binom{n}{j} \kappa_{X}(j+1)P_{n-j}(z).
\label{form-f1}
\end{equation}
\noindent
The function $f(z)$ may also be expressed as
\begin{equation}
f(z) = \sum_{j=0}^{n} \binom{n}{j} 
\mathbb{E} \left[ X_{1}(X_{1}+z)^{n-j} (Y + X_{2})^{j} \right] - 
\mathbb{E} X_{1} \, \mathbb{E} (z + X_{2})^{n}.
\end{equation}
\noindent
The relation $\mathbb{E} (Y + X_{2})^{j} = \delta_{j}$ holds since $X_{2}$ and 
$Y$ are conjugate random variables. This reduces the previous expression 
for $f$ to 
\begin{equation}
f(z) = \mathbb{E} X_{1}(X_{1}+z)^{n} - \mathbb{E} X_{1} \, P_{n}(z). 
\end{equation}
\noindent
This can be simplified using 
\begin{eqnarray*}
\mathbb{E} X_{1}(X_{1}+z)^{n} & = & 
\mathbb{E} (X_{1}+z)^{n+1} - z \mathbb{E} (X_{1}+z)^{n} \\
& = & P_{n+1}(z) - z P_{n}(z).
\end{eqnarray*}
\noindent
The function $f$ has been expressed as 
\begin{equation}
f(z) = P_{n+1}(z) - (z + \kappa_{X}(1)) P_{n}(z)
\label{form-f2}
\end{equation}
\noindent
using $\mathbb{E}(X) = \kappa_{X}(1)$. The recurrence for $P_{n}$ comes by 
comparing \eqref{form-f1} and \eqref{form-f2}. 

The second identity is obtained by replacing $X$ and $Y$ and remarking that 
$\kappa_{X}(p) = - \kappa_{Y}(p)$ and $\mathbb{E}(X+Y) = 0$, since $X$ and 
$Y$ are conjugate random variables. 
\end{proof}

\begin{theorem}
The hypergeometric Bernoulli polynomials $B_{n}^{(a,b)}(z)$ and the 
companion family $C_{n}^{(a,b)}(z)$ defined by
\begin{equation}
B_{n}^{(a,b)}(z) = \mathbb{E}(z+ \mathfrak{Z}_{a,b})^{n} \text{ and }
C_{n}^{(a,b)}(z) = \mathbb{E}(z+ \mathfrak{B}_{a,b})^{n}
\end{equation}
\noindent
satisfy the recurrences
\begin{equation}
\label{recu-forB}
B_{n+1}^{(a,b)}(z) - z B_{n}^{(a,b)}(z) = 
\sum_{j=0}^{n} \frac{n!}{(n-j)!} \zeta_{a,b}^{H}(j+1) B_{n-j}^{(a,b)}(z)
\end{equation}
\noindent
and 
\begin{equation}
C_{n+1}^{(a,b)}(z) - z C_{n}^{(a,b)}(z) = 
-\sum_{j=0}^{n} \frac{n!}{(n-j)!} \zeta_{a,b}^{H}(j+1) C_{n-j}^{(a,b)}(z).
\end{equation}
\end{theorem}
\begin{proof}
\noindent
The result now follows from Theorem \ref{thm-conjugate} and the cumulants 
for the beta distribution given in Example \ref{cumulants-beta}.
\end{proof}

Our last result provides a probabilistic approach to the linear recurrences 
for the hypergeometric zeta function. 

For a random variable $X$, the moments $\mathbb{E}X^{n}$ and its cumulants 
$\kappa_{X}(n)$ satisfy the relation \eqref{mom-cum}. This is now used to 
produce a linear recurrence for the hypergeometric zeta function.

\begin{theorem}
\label{thm-linear-2}
The hypergeometric zeta function $\zeta_{a,b}^{H}$ satisfies
\begin{equation}
(n-1)! \sum_{j=2}^{n} \frac{B_{n-j}^{(a,b)}}{(n-j)!} \zeta_{a,b}^{H}(j) = 
\frac{a}{a+b} B_{n-1}^{(a,b)} + B_{n}^{(a,b)}.
\end{equation}
\end{theorem}
\begin{proof}
Use the identity \eqref{mom-cum} to the random variable $\mathfrak{Z}_{a,b}$. 
Its moments are the hypergeometric Bernoulli numbers
\begin{equation}
\mathbb{E} \mathfrak{Z}_{a,b}^{p}  = B_{p}^{(a,b)}
\end{equation}
\noindent
and its cumulants are 
\begin{equation}
\kappa_{\mathfrak{Z}_{a,b}}(n) = 
\begin{cases}
(n-1)! \zeta_{a,b}^{H}(n), & \text{ for } 
n \geq 2 \\
- \frac{a}{a+b} & \text{ for } n = 1,
\end{cases}
\end{equation}
\noindent
since $\mathfrak{Z}_{a,b}$ and $\mathfrak{B}_{a,b}$ are conjugate random 
variables. A second proof is obtained by letting $z=0$ in \eqref{recu-forB}.
\end{proof}

\begin{note}
Surprisingly, the 
two linear recurrences for $\zeta_{a,b}^{H}$, given in Theorem 
\ref{thm-linear} and in Theorem \ref{thm-linear-2}  are different. For 
example, choosing 
$a=5, \, b=3$ these produce for $n=3$ the relations 
\begin{eqnarray*}
2 \zeta_{5,3}^{H}(3) + \frac{5}{4} \zeta_{5,3}^{H}(2) + \frac{1}{32} & = & 0 \\
2 \zeta_{5,3}^{H}(3) - \frac{5}{4} \zeta_{5,3}^{H}(2) - \frac{13}{384} & = & 0.
\end{eqnarray*}
\end{note}

\noindent
\textbf{Acknowledgements}. The work of the third author was 
partially supported by NSF-DMS 0070567. The first author is an 
undergraduate student and the second one is a graduate 
student at Tulane University, both partially supported by the same grant.


\begin{thebibliography}{10}

\bibitem{abramowitz-1972a}
M.~Abramowitz and I.~Stegun.
\newblock {\em Handbook of {M}athematical {F}unctions with {F}ormulas, {G}raphs
  and {M}athematical {T}ables}.
\newblock Dover, New York, 1972.

\bibitem{actor-1996a}
A.~Actor and I.~Bender.
\newblock The zeta function constructed from the zeros of the {B}essel
  function.
\newblock {\em J. {P}hys. {A}: {M}ath. {G}en.}, 29:6555--6580, 1996.

\bibitem{amdeberhan-2013e}
T.~Amdeberhan, V.~Moll, and C.~Vignat.
\newblock A probabilistic interpretation of a sequence related to the
  {N}arayana polynomials.
\newblock {\em Online {J}ournal of {A}nalytic {C}ombinatorics}, To appear,
  2013.

\bibitem{boas-1954a}
R.~P. Boas.
\newblock {\em Entire functions}.
\newblock Academic Press, New York, 1954.

\bibitem{carlitz-1961b}
L.~Carlitz.
\newblock The {S}taudt-{C}lausen theorem.
\newblock {\em Math. {M}ag.}, 34:131--146, 1961.

\bibitem{crandall-1996a}
R.~E. Crandall.
\newblock On the quantum zeta function.
\newblock {\em J. {P}hys. {A}: {M}ath. {G}en.}, 29:6795--6816, 1996.

\bibitem{deangelis-2009a}
V.~De~Angelis.
\newblock Stirling's series revisited.
\newblock {\em Amer. {M}ath. {M}onthly}, 116:839--843, 2009.

\bibitem{dinardo-2010a}
E.~Di~Nardo, P.~Petrullo, and D.~Senato.
\newblock Cumulants and convolutions via {A}bel polynomials.
\newblock {\em Europ. Journal of Comb.}, 31:1792--1804, 2010.

\bibitem{dilcher-2002b}
K.~Dilcher.
\newblock Bernoulli numbers and confluent hypergeometric functions.
\newblock In B.~C Berndt, N~Boston, H.~G. Diamond, A.~J. Hildebrand, and
  W.~Phillipp, editors, {\em Number {T}heory for the {M}illenium, {U}rbana,
  {IL}, 2000}, volume~I, pages 343--363. A.~K.~Peters, 2002.

\bibitem{dilcher-2002a}
K.~Dilcher and L.~Malloch.
\newblock Arithmetic properties of {B}ernoulli-{P}ad\'{e} numbers and
  polynomials.
\newblock {\em Journal of Number Theory}, 92:330--347, 2002.

\bibitem{elizalde-1995a}
E.~Elizalde.
\newblock {\em Ten {P}hysical {A}pplications of {S}pectral {Z}eta {F}unctions}.
\newblock Lecture Notes in Physics. Springer-Verlag, New York, 1995.

\bibitem{hassen-2008a}
A.~Hassen and H.~Nguyen.
\newblock Hypergeometric {B}ernoulli polynomials and {A}ppell sequences.
\newblock {\em Intern. {J}. {N}umber {T}heory}, 4:767--774, 2008.

\bibitem{hassen-2010a}
A.~Hassen and H.~Nguyen.
\newblock Hypergeometric zeta functions.
\newblock {\em Intern. {J}. {N}umber {T}heory}, 6:99--126, 2010.

\bibitem{hawkins-1983a}
J.~Hawkins.
\newblock {\em On a zeta function associated with {B}essel's equation}.
\newblock PhD thesis, University of Illinois, 1983.

\bibitem{howard-1967b}
F.~T. Howard.
\newblock A sequence of numbers related to the exponential function.
\newblock {\em Math. Comp.}, 34:599--615, 1967.

\bibitem{howard-1967a}
F.~T. Howard.
\newblock Some sequences of numbers related to the exponential function.
\newblock {\em Math. Comp.}, 34:701--716, 1967.

\bibitem{ireland-1990a}
K.~Ireland and M.~Rosen.
\newblock {\em A classical introduction to {N}umber {T}heory}.
\newblock Springer Verlag, 2nd edition, 1990.

\bibitem{lasallem-2012a}
M.~Lassalle.
\newblock Two integer sequences related to {C}atalan numbers.
\newblock {\em J. Comb. Theory Ser. A}, 119:923--935, 2012.

\bibitem{lozier-2003a}
D.~W. Lozier.
\newblock The {NIST} {D}igital {L}ibrary of {M}athematical {F}unctions
  {P}roject.
\newblock {\em Ann. Math. Art. Intel.}, 38(1-3):105--119, 2003.

\bibitem{smith-1995a}
P.~J. Smith.
\newblock A recursive formulation of the old problem of obtaining moments from
  cumulants and viceversa.
\newblock {\em The {A}merican {S}tatistician}, 49:217--218, 1995.

\bibitem{stolarski-1985a}
K.~Stolarski.
\newblock Singularities of {B}essel-zeta functions and {H}awkins' polynomials.
\newblock {\em Mathematika}, 32:96--103, 1985.

\end{thebibliography}
\end{document}